%% file: main-arxiv.tex
\documentclass[letterpaper, 10 pt, conference]{ieeeconf}  

\IEEEoverridecommandlockouts                              
\title{\LARGE \bf
Online Adversarial Stabilization of \\Unknown Linear Time-Varying Systems
}

\author{Jing Yu, Varun Gupta, and Adam Wierman
\thanks{This work was supported by Caltech/Amazon AWS AI4Science
fellowships and the National Science Foundation under grants CNS-2146814,
CPS-2136197, CNS-2106403, ECCS-2200692, and NGSDI-2105648. 
J. Yu and A. Wierman are with the Department of Computing and Mathematical Sciences, California Institute of Technology {\tt\small \{jing, adamw\}@caltech.edu}. V. Gupta is with the Booth School of Business, University of Chicago {\tt\small guptav@uchicago.edu}}
}

\input{cdc-preamble.tex}

\begin{document}

\maketitle
\thispagestyle{empty}
\pagestyle{empty}

\begin{abstract}
This paper studies the problem of online stabilization of an unknown discrete-time linear time-varying (LTV) system under bounded non-stochastic (potentially adversarial) disturbances. 
We propose a novel algorithm based on convex body chasing (CBC). Under the assumption of infrequently changing or slowly drifting dynamics, the algorithm guarantees bounded-input-bounded-output stability in the closed loop. Our approach avoids system identification and applies, with minimal disturbance assumptions, to a variety of LTV systems of practical importance.  We demonstrate the algorithm numerically on examples of LTV systems including Markov linear jump systems with finitely many jumps.
\end{abstract}

\section{Introduction}
Learning-based control of linear-time invariant (LTI) systems in the context of linear quadratic regulators (LQR) has seen considerable progress. However, many real-world systems are time-varying in nature. For example, the grid topology in power systems can change over time due to manual operations or unpredictable line failures \cite{ deka2017structure}. Therefore, there is increasing recent interest in extending learning-based control of LTI systems to the linear time-varying (LTV) setting \cite{gradu2020adaptive,nortmann2020data,luo2022dynamic,lin2022online, minasyan2021online}. 

LTV systems are widely used to approximate and model real-world dynamical systems such as robotics \cite{tedrake2009underactuated} and autonomous vehicles \cite{falcone2008linear}. In this paper, we consider LTV systems with dynamics of the following form: 
\begin{equation}
\label{eq:ltv}
    x_{t+1} = A_t x_t + B_t u_t + w_t,  
\end{equation}
where $x_t \in \RR^n$, $u_t \in \RR^m$ and $w_t$ denotes the state, the control input, and the bounded and potentially adversarial disturbance, respectively.
We use $\theta_t = [A_t\ B_t]$ to succinctly denote
the system matrices at time step $t$. 

On the one hand, \textbf{offline control} design for LTV systems is well-established in the setting where the underlying LTV model is \textit{known} \cite{tsakalis1993linear, toth2010modeling,mohammadpour2012control, zhang2019sampled, mojgani2020stabilization}. Additionally, recent work has started focusing on regret analysis and non-stochastic disturbances for known LTV systems \cite{gradu2020adaptive, goel2023regret}.

On the other hand, \textbf{online control} design for LTV systems where the model is \textit{unknown} is more challenging. 
Historically, there is a rich body of work on adaptive control design for LTV systems \cite{tsakalis1987adaptive, slotine1986adaptive, marino2003adaptive}. Also related is the system identification literature for LTV systems \cite{verhaegen1995class, bamieh2002identification,sarkar2019nonparametric}, which estimates the (generally assumed to be stable) system to allow the application of the offline techniques.

In recent years, the potential to leverage modern data-driven techniques for controller design of unknown linear systems has led to a resurgence of work in both the LTI and LTV settings. There is a growing literature on ``learning to control'' unknown LTI systems under stochastic or no noise \cite{dean2019safely,talebi2021regularizability, lale2022reinforcement}. 
Learning under bounded and potentially adversarial noises poses additional challenges, but online stabilization \cite{yu2023online} and regret \cite{chen2021black} results have been obtained.  

In comparison, there is much less work on learning-based control design for unknown LTV systems. One typical approach, exemplified by \cite{nortmann2020data, rotulo2022online, baros2022online}, derives stabilizing controllers under the assumption that \textit{offline} data representing the input-output behavior of \eqref{eq:ltv} is available and therefore an \textit{offline} stabilizing controller can be pre-computed. Similar \textit{finite-horizon} settings where the algorithm has access to offline data \cite{pang2018data}, or can \textit{iteratively} collect data \cite{liu2016batch} were also considered. In the context of \textit{online} stabilization, i.e., when offline data is not available, work has derived stabilizing controllers for LTV systems through the use of predictions of $\theta_t$, e.g., \cite{qu2021stable}. Finally, another line of work focuses on designing regret-optimal controllers for LTV systems \cite{ouyang2017learning, minasyan2021online, luo2022dynamic, lin2022online, han2022learning}. However, with the exception of \cite{qu2021stable}, existing work on \textit{online} control of unknown LTV systems share the common assumption of either of open-loop stability or knowledge of an offline stabilizing controller. Moreover, the disturbances are generally assumed to be zero or stochastic noise independent of the states and inputs.  

In this paper, we propose an online algorithm for stabilizing unknown LTV systems under bounded, potentially adversarial disturbances.  Our approach uses convex body chasing (CBC), which is an online learning problem where one must choose a sequence of points within sequentially presented convex sets with the aim of minimizing the sum of distances between the chosen points \cite{sellke2020chasing, argue2021chasing}. CBC has emerged as a promising tool in online control, with most work making connections to a special case called \textit{nested} convex body chasing (NCBC), where the convex sets are sequentially nested within the previous set \cite{bansa2018nested, bubeck2020chasing}. In particular, \cite{ho2021online} first explored the use of NCBC for learning-based control of time-invariant nonlinear systems. NCBC was also used in combination with System Level Synthesis to design a distributed controller for networked systems \cite{yu2023online} and in combination with model predictive control \cite{yeh2022robust} for LTI system control as a promising alternative to system identification based methods. However, this line of work depends fundamentally on the time invariance of the system, which results in nested convex sets. LTV systems do not yield nested sets and therefore represent a significant challenge.

This work addresses this challenge and presents a novel online control scheme (Algorithm~\ref{alg:main}) based on CBC (non-nested) techniques that guarantees bounded-input-bounded-output (BIBO) stability as a function of the total model variation $\sum_{t=1}^{\infty} \norm{\theta_{t} - \theta_{t-1}}$, without predictions or offline data under {bounded and potentially adversarial} disturbances for unknown LTV systems (\Cref{thrm:iss}). This result implies that when the total model variation is finite or growing sublinearly, BIBO stability of the closed loop is guaranteed (\Cref{cor:stability-bounded,cor:stability-sublinear}).  
In particular, our result depends on a refined analysis of the CBC technique (\Cref{lem:partial-path}) and is based on the perturbation analysis of the Lyapunov equation. This contrasts with previous NCBC-based works for time-invariant systems, where the competitive ratio guarantee of NCBC directly applies and the main technical tool is the robustness of the model-based controller, which is a proven using a Lipschitz bound of a quadratic program in \cite{yu2023online} and is directly assumed to exist in \cite{ho2021online}.

We illustrate the proposed algorithm via numerical examples in \Cref{sec:simulation} to corroborate the stability guarantees. We demonstrate how the proposed algorithm can be used for data collection and complement data-driven methods like \cite{baros2022online,nortmann2020data, pang2018data}. Further, the numerics highlight that the proposed algorithm can be efficiently implemented by leveraging the linearity of \eqref{eq:ltv}  despite the computational complexity of CBC algorithms in general (see \Cref{subsec:implementation} for details).

\textbf{Notation. } We use $\mathbb{S}^{n-1}$ to denote the unit sphere in $\RR^n$ and $\mathbb{N}_+$ for positive integers. For $t,\, s\in \mathbb{N}_+$, we use $[t:s]$ as shorthand for the set of integers $\{t,\,t+1,\,\ldots,s\}$ and $[t]$ for $\{1,\,2,\,\ldots,t\}$. Unless otherwise specified, $\norm{\cdot}$ is the operator norm. We use $\rho(\cdot)$ for the spectral radius of a matrix.


\section{Preliminaries}
In this section, we state the model assumptions underlying our work and review key results for convex body chasing, which we leverage in our algorithm design and analysis. 

\subsection{Stability and model assumptions}
\label{sec:model}

We study the dynamics in \eqref{eq:ltv} and make the following standard assumptions about the dynamics.
\vspace{-0.3cm}
\begin{assumption}
\label{assum:noise_norm}
    The disturbances are bounded: $\infnorm{w_t} \leq W$ for all $t \geq 0$.
\end{assumption}
\vspace{-0.5cm}
\begin{assumption}
\label{assum:compact_set}
    The unknown time-varying system matrices $\sysm$ belong to a known (potentially large) polytope $\Theta$ such that $\theta_t \in \Theta$ for all $t$. Moreover, there exists $\kappa>0$ such that $\norm{\theta} \leq \kappa$ and  $\theta$ is stabilizable for all $\theta \in \Theta$.
\end{assumption}

Bounded and non-stochastic (potentially adversarial) disturbances is a common model both in the online learning and control problems \cite{ramasubramanian2020safety, bai1995membership}. Since we make no assumptions on how large the bound $W$ is, \Cref{assum:noise_norm} models a variety of scenarios, such as bounded and/or correlated stochastic noise, state-dependent disturbances, e.g., the linearization and discretization error for nonlinear continuous-time dynamics, and potentially adversarial disturbances. \Cref{assum:compact_set} is standard in learning-based control, e.g. \cite{cohen2019learning, agarwal2019online}.

We additionally assume there is a quadratic known cost function of the state and control input at every time step $t$ to be minimized, e.g. $x_t^\top Q x_t + u_t^\top R u_t$, with $Q,\,R \succ 0$. For a given LTI system model $\theta=[A\ B]$ and cost matrices $Q,\,R$, we denote $K = \textsc{LQR}(\theta;Q,R)$ as the optimal feedback gain for the corresponding infinite-horizon LQR problem.

\begin{remark}
Representing model uncertainty as convex compact parameter sets where every model is stabilizable is not always possible. In particular, if a parameter set $\Theta$ has a few singular points where $(A,B)$ loses stabilizability such as when $B=0$, a simple heuristic is to ignore these points in the algorithm since we assume the underlying true system matrices $\theta_t$ must be stabilizable.
\end{remark}


\subsection{Convex body chasing}
Convex Body Chasing (CBC) is a well-studied online learning problem \cite{bansa2018nested,bubeck2020chasing}. At every round $t \in \mathbb{N}_+$, the player is presented a convex body/set $\mathcal{K}_t \subset \mathbb{R}^n$.
The player selects a point $q_t \in \mathcal{K}_t$ with the objective of minimizing the cost defined as the total path length of the selection for $T$ rounds, e.g., $\sum_{t=1}^T \|q_{t} - q_{t-1}\|$ for a given initial condition $q_0\not \in \mathcal{K}_1$. 
There are many known algorithms for the CBC problem with a \textit{competitive ratio} guarantee such that the cost incurred by the algorithm is at most a constant factor from the total path length incurred by the offline optimal algorithm which has the knowledge of the entire sequence of the bodies. We will use CBC to select $\theta_t$'s that are consistent with observed data.

\subsubsection{The nested case}
A special case of CBC is the \textit{nested} convex body chasing (NCBC) problem, where $\mathcal{K}_t \subseteq \mathcal{K}_{t-1}$. A known algorithm for NCBC is to select the \textit{Steiner point} of $\mathcal{K}_t$ at $t$ \cite{bubeck2020chasing}. The Steiner point of a convex set $\mathcal{K}$ can be interpreted as the average of the extreme points of $\mathcal{K}$ and is defined as 
$\st(\mathcal{K}):= \mathbb{E}_{v: \|v\|\leq 1} \left[ g_{\mathcal{K}}(v) \right]$, where $g_{\mathcal{K}}(v) := \text{argmax}_{x\in\mathcal{K}} v^\top x $ and the expectation is taken with respect to the uniform distribution over the unit ball. The intuition is that Steiner point remains ``deep" inside of the (nested) feasible region so that when this point becomes infeasible due to a new convex set, this convex set must shrink considerably, which indicates that the offline optimal must have moved a lot. 
Given the initial condition $q_0 \not \in \mathcal{K}_1$, the Steiner point selector achieves competitive ratio of $\mathcal{O}(n)$ against the offline optimal such that for all $T \in \mathbb{N_+}$, $\sum_{t=1}^T \|\st(\mathcal{K}_{t})- \st(\mathcal{K}_{t-1})\| \leq \bigO(n) \cdot \text{OPT}$, where $\text{OPT}$ is the offline optimal total path length.
There are many works that combine the Steiner point algorithm for NCBC with existing control methods to perform learning-based online control for LTI systems, e.g.,  \cite{yu2023online, ho2021online,yeh2022robust}.

\subsubsection{General CBC}
\label{subsec:cbc}
For general CBC problems, we can no longer take advantage of the nested property of the convex bodies. One may consider naively applying NCBC algorithms when the convex bodies happen to be nested and restarting the NCBC algorithm when they are not. However, due to the myopic nature of NCBC algorithms, which try to remain deep inside of each convex set, they no longer guarantee a competitive ratio when used this way.
Instead, \cite{sellke2020chasing} generalizes ideas from NCBC and proposes an algorithm that selects the \textit{functional Steiner point} of the \textit{work function}.

\begin{definition}[Functional Steiner point]\label{defn:supp_fn_fn}
For a convex function $f:\RR^n \to \RR$, the functional Steiner point of $f$ is
\begin{equation}
\label{eq:steiner}
    \st(f) = -n \cdot \dashint_{v:\norm{v}= 1}  f^*(v) \, v \, dv,
\end{equation}
where $\dashint_{x\in\mathcal{S}} f(x) dx$ denotes the normalized value $\frac{\int_{x\in \mathcal{S}} f(x) dx}{\int_{x\in\mathcal{S}} 1 dx}$ of $f(x)$ on the set $\mathcal{S}$, and 
\begin{equation}
    \label{eq:conjugate}
    f^*(v):= \text{inf}_{x\in\RR^n}  f(x) - \inner{x}{v}
\end{equation} 
is the Fenchel conjugate of $f$.
\end{definition}

The CBC algorithm selects the functional Steiner point of the \textit{work function}, which records the smallest cost required
to satisfy a sequence of requests while ending in a given state, thereby encapsulating information about the offline-optimal cost for the CBC problem.
\begin{definition}[Work function]
Given an initial point $q_0 \in \RR^n$, and convex sets $\mathcal{K}_1,\ldots, \mathcal{K}_t \subset \RR^n$, the work function at time step $t$  evaluated at a point $x \in \RR^n$ is given by:
\begin{equation}
\label{eq:work-function}
    \omega_t(x) = \min_{q_s \in \mathcal{K}_s} \norm{x - q_t} + \sum_{s=1}^t \norm{ q_s - q_{s-1} }.
\end{equation}
\end{definition}
Importantly, it is shown that the functional Steiner points of the work functions are valid, i.e., $\st(\omega_t) \in \mathcal{K}_t$ for all $t$ \cite{sellke2020chasing}. On a high level, selecting the functional Steiner point of the work function helps the algorithm stay competitive against the currently estimated offline optimal cost via the work function, resulting in a competitive ratio of $n$ against the offline optimal cost ($\text{OPT}$) for general CBC problems,
\begin{equation}
\label{eq:competitive}
\sum_{t=1}^T \|\st(\omega_{t})- \st(\omega_{t+1})\| \leq n \cdot \text{OPT}.
\end{equation}

Given the non-convex nature of \eqref{eq:steiner} and \eqref{eq:work-function}, we note that, in general, it is challenging to compute the functional Steiner point of the work function. However, in the proposed algorithm, we are able to leverage the linearity of the LTV systems and numerically approximate both objects with efficient computation in \Cref{subsec:implementation}.

\section{Main Results}
\label{sec:main_result}

We present our proposed online control algorithm to stabilize the unknown LTV system \eqref{eq:ltv} under bounded and potentially adversarial disturbances in Algorithm~\ref{alg:main}. After observing the latest transition from $x_{t},\,u_{t}$ to $x_{t+1}$ at $t+1$ according to \eqref{eq:ltv} (line~\ref{algoline:dynamics}), the algorithm constructs the set of all feasible models $\widehat \theta_t$'s (line~\ref{algoline:set}) such that the model is \textit{consistent} with the observation, i.e., there exists an admissible disturbance $\hat w_t$ satisfying \Cref{assum:noise_norm} such that the state transition from $x_{t},\,u_{t}$ to $x_{t+1}$ can be explained by the tuple ($\widehat \theta_t$, $\hat w_t$). We call this set the \textit{consistent model set} $\mathcal{P}_t$ and we note that the unknown true dynamics $\theta_t = [A_t \ B_t]$ belongs to $\mathcal{P}_t$. The algorithm then selects a \textit{hypothesis} model out of the consistent model set $\mathcal{P}_t$ using the CBC algorithm by computing the functional Steiner point \eqref{eq:steiner} of the work function \eqref{eq:work-function} with respect to the history of the consistent parameter sets $\mathcal{P}_1,\, \ldots,\, \mathcal{P}_t$ (line~\ref{algoline:cbc}). In particular, we present an efficient implementation of the functional Steiner point chasing algorithm in \Cref{subsec:implementation} by taking advantage of the fact that $\mathcal{P}_t$'s are polytopes that can be described by intersection of half-spaces. The implementation is summarized in Algorithm~\ref{alg:cbc}. Based on the selected hypothesis model $\hat \theta_t$, a certainty-equivalent LQR controller is synthesized (line~\ref{algoline:lqr}) and the state-feedback control action is computed (line~\ref{algoline:ut}). 

Note that, by construction, at time step $t \in \mathbb{N}_+$ we perform certainty-equivalent control $\widehat K_{t-1}$ based on a hypothesis model $\hat \theta_{t-1}$ computed using retrospective data, even though the control action ($u_t = \widehat K_{t-1} x_t$) is applied to the dynamics ($\theta_t$) that we do not yet have any information about. In order to guarantee stability, we would like for $\widehat K_{t-1}$ to be stabilizing the ``future'' dynamics ($\theta_t$). This is the main motivation behind our choice of the CBC technique instead of regression-based techniques for model selection. Thanks to the competitive ratio guarantee \eqref{eq:competitive} 
 of the functional Steiner point selector, when the true model variation is ``small,'' our previously selected hypothesis model will stay ``consistent'' in the sense that $\widehat K_{t-1}$ can be stabilizing for $\theta_t$ despite the potentially adversarial or state-dependent disturbances. On the other hand, when the true model variation is ``large,'' $\widehat K_{t-1}$ does not stabilize $\theta_t$, and we see growth in the state norm. Therefore, our final state bound is in terms of the total variation of the true model. 

We show in the next section that, by drawing connections between the stability of the closed-loop system and the path length cost of the selected hypothesis model via CBC, we are able to stabilize the unknown LTV system without any identification requirements, e.g., the selected hypothesis models in Algorithm~\ref{alg:main} need not be close to the true models. It is observed that even in the LTI setting, system identification can result in large-norm transient behaviors with numerical stability issues if the underlying unknown system is open-loop unstable or under non-stochastic disturbances; thus motivating the development of NCBC-based online control methods \cite{chen2021black, yu2023online, ho2021online}.  In the LTV setting, it is not sufficient to use NCBC ideas due to the time-variation of the model; however, the intuition for the use of CBC is similar.  In fact, it can be additionally beneficial to bypass identification in settings where the true model is a moving target, thus making identification more challenging. We illustrate this numerically in \Cref{sec:simulation}.

\SetKwInOut{Initialize}{Initialize}
\SetKwInOut{Input}{Input}
\begin{algorithm} 
\LinesNumbered
\SetAlgoLined
\DontPrintSemicolon
\KwIn{$W >0$, $\Theta \subset \RR^{n\times (n+m)}$} 
\Initialize{$u_0 = 0$, $\hat \theta_0 \in \Theta$ }
    \For{$t+1 = 1,2, \ldots$}{ 
     Observe $x_{t+1}$ \nllabel{algoline:dynamics} \\
     Construct consistent set $\mathcal{P}_t := 
     \left\{ \theta = [A , B] : \infnorm{x_{t+1} - A x_t - B u_t} \leq W\right\}  \cap \Theta$ \nllabel{algoline:set}\\
     Select hypothesis model $ \widehat{\theta}_{t} \leftarrow \textsc{CBC}(\{\mathcal{P}_s\}_{s=1}^t; \hat \theta_0 )$ \label{algoline:cbc}\\
    Synthesize controller $\widehat{K}_{t} \leftarrow \textsc{LQR}\left(\widehat{\theta}_{t} ; Q, R\right)$ \nllabel{algoline:lqr}\\
     Compute feedback control input $u_{t+1} = \widehat{K}_{t} x_{t+1}$ \nllabel{algoline:ut}}
    \caption{\textsc{Unknown LTV stabilization}}
    \label{alg:main}
\end{algorithm}

\begin{algorithm}
  \caption{\textsc{CBC}}  
  \label{alg:cbc}
    \LinesNumbered
\SetAlgoLined
\KwIn{$\mathcal{P}_1$, $\ldots$, $\mathcal{P}_t$, $\hat \theta_0$, $N$}
\KwOut{$\hat \theta_t$} 
\For{$k=0,1, \ldots N$}{
Sample $v_i$ uniformly from $\mathbb{S}^{n-1}$\\
$ h_i \leftarrow \eqref{eq:socp}$\\
}
$\hat \theta_t \leftarrow  \textsf{proj}_{\Theta \cap \mathcal{P}_t}\left(-\frac{n}{N} \sum_{i=1}^N h_i v_i\right)$
\end{algorithm}

\subsection{Stability Analysis}

The main result of this paper is the BIBO stability guarantee for Algorithm~\ref{alg:main} in terms of the true model variation and the disturbance bound. We sketch the proof in this section and refer \Cref{appendix:main_proof} for the formal proof. This result depends on a refined analysis of the competitive ratio for the functional Steiner point chasing algorithm introduced in \cite{sellke2020chasing}, which is stated as follows.
\begin{lemma}[Partial-path competitive ratio]
    \label{lem:partial-path}
    For $t\in \mathbb{N}_+$, let $s,\,e \in [t]$ and $s<e$, and let $\Theta\subset \RR^n$ be a convex compact set. Denote $ \hat{\Delta}_{[s,e]}:=\sum_{\tau=s+1}^{e} \norm{ \st(\omega_{\tau}) -  \st(\omega_{\tau-1}) }_F$ as the partial-path cost of the functional Steiner point selector during interval $[s,e]$ and $\{\textsc{OPT}_\tau\}_{\tau = 1}^{t}$ as the (overall) offline optimal selection for $\mathcal{K}_1,\,\ldots ,\, \mathcal{K}_t \subset \Theta$. The functional Steiner point chasing algorithm has the following competitive ratio,
    \begin{align*}
  \hat{\Delta}_{[s,e]} &\leq n \left( \textsf{dia}(\Theta) + 2\kappa + \sum_{\tau=s+1}^{e} \norm{\textsc{OPT}_{\tau} - \textsc{OPT}_{\tau-1}}_F\right).
\end{align*}
on interval $[s,e]$, where $\textsf{dia}(\Theta):= \max_{\theta_1,\, \theta_2 \in \Theta} \norm{\theta_1 - \theta_2}_F $ denotes the diameter of $\Theta$ and $\kappa := \max_{\theta \in \Theta} \norm{\theta}_F$.
\end{lemma}
\begin{proof}
    See \Cref{appedix:partial-path}.
\end{proof}

\begin{theorem}[BIBO Stability]
\label{thrm:iss}
Under Assumption~\ref{assum:noise_norm} and \ref{assum:compact_set}, the closed loop of \eqref{eq:ltv} under Algorithm~\ref{alg:main} is BIBO stable such that for all $t \geq 0$, 
\begin{equation*}
     \norm{x_{t}} \leq W \cdot c_1 \sum_{s=0}^{t-2} c_2^{\Delta_{[s,t-1]}} \rho_L^{t-s}
\end{equation*}
where ${\Delta}_{[s,t-1]}:=\sum_{\tau=s+1}^{t-1} \norm{ {\theta}_{\tau} -  {\theta}_{\tau-1} }_F$ is the true model variation, $W$ is the disturbance bound, and $c_1, \, c_2 >0,\, \rho_L\in(0,1) $ are constants that depend on the system-theoretical quantities of the worst-case model in the parameter set $\Theta$.
\end{theorem}
\hspace{.1in} \textit{Proof Sketch:}
At a high level, the structure of our proof is as follows. We first use the fact that our time-varying feedback gain $\widehat K_t$ is computed according to a hypothesis model from the \textit{consistent} model set. Therefore, we can characterize the closed-loop dynamics in terms of the consistent models $\hat \theta_t$ and $\widehat K_t$. 
Specifically, consider a time step $t$ where we take the action $u_t = \widehat{K}_{t-1} x_t$ after observing $x_t$. Then, we observe $x_{t+1} = A_t x_t + B_t u_t + w_t$ and select a new hypothesis model $\widehat{\theta}_{t} = [\widehat{A}_{t} \ \widehat{B}_{t}]$ that is consistent with this new observation. Since we have selected a consistent hypothesis model, there is some admissible disturbance $\hat w_t$ satisfying \Cref{assum:noise_norm} such that 
\begin{align*}
x_{t+1}  &= \left( A_t + B_t \widehat{K}_{t-1}  \right) x_t + w_t = \left( \widehat{A}_{t} + \widehat{B}_{t} \widehat{K}_{t-1}  \right) x_t + \hat{w}_t. 
\end{align*}
Without loss of generality, we assume initial condition $x_0 = 0$. We therefore have
\begin{align}
\label{eq:closed-loop}
    x_t =  \hat w_{t-1} + \sum_{s = 0}^{t-2}  \prod_{\tau \in [t-1 : s+1] } \left( \widehat{A}_{\tau} + \widehat{B}_{\tau} \widehat{K}_{\tau-1} \right)    \hat{w}_s.
\end{align}
We have two main challenges in bounding $\norm{x_t}$ in \eqref{eq:closed-loop}:
\begin{enumerate}
\item $\widehat K_t$ is computed using $\hat \theta_t$ in Algorithm~\ref{alg:main}, but is applied to the next time step $\hat \theta_{t+1}$. While we know $\rho(\widehat{A}_{t} + \widehat{B}_{t} \widehat{K}_{t}) < 1$, in \eqref{eq:closed-loop} we have $\widehat{K}_{t-1}$ instead of $\widehat{K}_{t}$. 
\item Naively applying submultiplicativity of the operator norm for \eqref{eq:closed-loop} results in bounding $\norm{ \left( \widehat{A}_{\tau} + \widehat{B}_{\tau} \widehat{K}_{\tau-1} \right) }$. However, even if $\widehat K_{t-1}$ satisfies $\rho(\widehat{A}_{t} + \widehat{B}_{t} \widehat{K}_{t}) < 1$, in general the operator norm can be greater than 1. 
\end{enumerate}
To address the first challenge, our key insight is that by selecting hypothesis models via CBC technique, in any interval where the true model variation is small, our selected hypothesis model also vary little. Specifically, by \Cref{lem:partial-path}, we can bound the partial-path variation of the selected hypothesis models with the true model partial-path variation $\Delta_{[s,e]}$ as follows.
    \begin{align}
  \hat{\Delta}_{[s,e]} &\leq n \left( \textsf{dia}(\Theta) + 2\kappa + \sum_{\tau=s}^{e-1} \norm{\textsc{OPT}_{\tau+1} - \textsc{OPT}_{\tau}}_F\right) \nonumber \\
  &\leq n \left( \textsf{dia}(\Theta) + 2\kappa + \Delta_{[s,e]}\right). \label{eq:chase}
\end{align}
where $\Theta$ and $\kappa$ are from \Cref{assum:compact_set}.
A consequence of \eqref{eq:chase} is that, during intervals where the true model variation is small, we have $\left(\widehat{A}_{t}+\widehat{B}_t \widehat{K}_{t-1} \right) \approx \left(\widehat{A}_{t}+\widehat{B}_t \widehat{K}_{t} \right)$. 

For the second challenge, we leverage the concept of sequential strong stability \cite{cohen2018online}, which allows bounding $\norm{\prod_{\tau \in [t-1 : s+1] } \left( \widehat{A}_{\tau} + \widehat{B}_{\tau} \widehat{K}_{\tau-1} \right)  } $ approximately with $ \prod_{\tau \in [t-1 : s+1] } \rho \left( \widehat{A}_{\tau} + \widehat{B}_{\tau} \widehat{K}_{\tau} \right)  $ times $\bigO\left( \exp(\Delta_{[s,t-1]}) \right)$. 

We now sketch the proof. The helper lemmas are summarized in \Cref{appendix:auxillary} and the formal proof can be found in \Cref{appendix:main_proof}.
Consider $L_t,\, H_t \in \RR^{n\times n}$ with $H_t\succ 0$ such that
\begin{align*}
    \widehat{A}_t + \widehat{B}_t \widehat{K}_{t-1} &:= H_t^{1/2}  L_t  H_t^{-1/2}.
\end{align*}
We use $I_s$ as shorthand for the interval $[t-1:s+1]$. Then each summand in \eqref{eq:closed-loop} can be bounded as
\begin{align}
    & \norm{ \prod_{\tau\in I_s} \left(\widehat A_\tau + \widehat B_\tau \widehat K_{\tau-1}\right) } \nonumber \\
    &\leq \underbrace{\norm{H_{t-1}^{1/2}} \norm{H_{s+1}^{-1/2}}}_{(a)}
        \underbrace{\prod_{k\in I_{s+1}} \norm{H_{k}^{-1/2}  H_{k-1}^{1/2}}}_{(b)}   
         \underbrace{\prod_{\tau\in I_s} \norm{L_{\tau}}}_{(c)} 
          \label{eqn:bound_norm_A}
\end{align}

Therefore showing BIBO stability comes down to bounding individual terms in \eqref{eqn:bound_norm_A}. In particular we will show that by selecting appropriate $H_t$ and $L_t$, term (a) is bounded by a constant $C_H$ that depends on system theoretical properties of the worst-case parameter in $\Theta$. For (b) and (c), we isolate the instances when \begin{equation}
\label{eq:small-movement}
\norm{\widehat \theta_t - \widehat \theta_{t-1}}_F \leq \epsilon 
\end{equation}
for some chosen $\epsilon >0$.
For instances where \eqref{eq:small-movement} holds, we use the perturbation analysis of the Lyapunov equation involving the matrix $ \widehat{A}_{t} + \widehat{B}_{t} \widehat{K}_{t-1}$ (\Cref{lem:HH_bound} for (b) and \Cref{lem:AHA_bound} for (c)) to bound (b) and (c) in terms of the partial-path movement of the selected parameters $\widehat \Delta_{[s,e]}:= \sum_{\tau=s+1}^{e} \norm{ \st(\omega_{\tau+1}) -  \st(\omega_{\tau}) }_F$. Specifically, \Cref{lem:HH_bound} implies 
\begin{align}
\label{eqn:bound_HH_main}
    \norm{H_t^{-1/2}H_{t-1}^{1/2}} &\leq
    \begin{cases}
        e^{\frac{\beta \norm{\widehat \theta_t - \widehat \theta_{t-1}}_F  }{2}}, & \text{if \eqref{eq:small-movement} holds}  \\
        \bar{H} & \mbox{otherwise},
    \end{cases}
\end{align}
where $\beta,\,\bar{H}>1$ are constants. We also show that from \Cref{lem:AHA_bound}, 
\begin{equation}
    \begin{aligned}
\label{eqn:bound_L_main}
    \norm{L_t} & \leq 
    \begin{cases}
    \rho_L & \text{if \eqref{eq:small-movement} holds} \\
    \bar{L} & \mbox{otherwise},
    \end{cases}
\end{aligned}
\end{equation}
for $\rho_L \in (0,1)$ and $\bar{L}>1$ a constant.

We now plug \eqref{eqn:bound_HH_main} and \eqref{eqn:bound_L_main} into  \eqref{eqn:bound_norm_A}.
Denote by $n_{[s,t]}$ the number of pairs $(\tau,\tau-1)$ with $s+1 \leq \tau \leq t-1$ where \eqref{eq:small-movement} fails to hold. Let ${\Delta}_{[s,e]} := \sum_{\tau = s+1}^e \norm{ \theta_{\tau} -  \theta_{\tau -1}}_F$ be the true model  partial-path variation. 
Then \eqref{eqn:bound_norm_A} can be bounded as
\begin{align*}
    &  \norm{\prod_{\tau\in [t-1:s+1]} \left(\widehat A_\tau + \widehat B_\tau \widehat K_{\tau-1}\right) }\\
    & \leq C_H \cdot \bar{H}^{n_{[s,t]}} \cdot e^{\frac{\beta \hat{\Delta}_{[s+1,t-1]} }{2}} \cdot \bar{L}^{n_{[s,t]}} \cdot \rho_L^{t-s-\hat n_{[s,t]}-1}
    \\
    &\leq C_H\left( \frac{\bar{L}\bar{H}}{\rho_L} \right)^{\frac{\hat{\Delta}_{[s,t-1]}}{\epsilon_*}}  e^{ \frac{\beta \hat{\Delta}_{[s+1,t-1]}}{2} } \cdot  \rho_L^{t-s-1}\\
    &\leq C_H \left( \frac{\bar{L}\bar{H}}{\rho_L} \right)^{\frac{ 
 \bar{n} \left( \mathsf{dia}(\Theta)+ 2\kappa + {\Delta}_{[s,t-1]} \right) }{\epsilon_*}}  e^{ \frac{\beta \bar{n} \left( \mathsf{dia}(\Theta)+ 2\kappa + {\Delta}_{[s+1,t-1]} \right)}{2} } \cdot  \rho_L^{t-s-1}\\
    & =: c \cdot c_2^{\Delta_{[s,t-1]}} \rho_L^{t-s},
\end{align*}
for constants $c, \,c_2$ and $\bar{n} := n(n + m)$ for the dimension of the parameter space for $A_t,\,B_t$. In the second inequality, we used the observation that 
$ n_{[s,t]} \leq \frac{\hat{\Delta}_{[s,t-1]}}{\epsilon}$ and in the last inequality we used \Cref{lem:partial-path}. Combined with \eqref{eq:closed-loop} and \Cref{assum:noise_norm}, this proves the desired bound.
\hfill $\blacksquare$

An immediate consequence of \Cref{thrm:iss} is that when the model variation in \eqref{eq:ltv} is bounded or sublinear, Algorithm~\ref{alg:main} guarantees BIBO stability.  This is summarized below.
\begin{corollary}[Bounded variation]
\label{cor:stability-bounded}
Suppose \eqref{eq:ltv} has model variation $\Delta_{[0,t]} \leq M$ for a constant $M$. Then,
\begin{align*}
   \sup_t \norm{x_{t}} &\leq \frac{c_1 \cdot c_2^M}{1-\rho_L}. 
\end{align*}
\end{corollary}

\begin{corollary}[Unbounded but sublinear variation]
\label{cor:stability-sublinear}
Let $\alpha\in(0,1)$ and $t\in \mathbb{N}_+$. Suppose \eqref{eq:ltv} is such that for each 
$k \leq t$, $\Delta_{[k,k+1]} \leq \delta_t := 1/{t^{(1-\alpha)}}$, implying a total model variation $\Delta_{[0,t]} = \bigO(t^{\alpha})$.
Then for large enough $t$, $\rho_L c_2^{\delta_t}\leq \frac{1+\rho_L}{2}$, and therefore
\begin{align*}
\norm{x_k}\leq c_1 \sum_{i=0}^k \left( \rho_L c_2^{\delta_t} \right)^i \leq  \frac{2c_1}{1-\rho_L}.
\end{align*}
\end{corollary}

\Cref{cor:stability-bounded} can be useful for scenarios where the mode of operation of the system changes infrequently and for systems such that $\theta(t) \rightarrow \theta^\star$ as $t \rightarrow \infty$ \cite{hahn1967stability}. As an example, consider power systems where a prescribed set of lines can potentially become disconnected from the grid and thus change the grid topology. \Cref{cor:stability-sublinear} applies to slowly drifting systems \cite{amato1993new}.

\subsection{Efficient implementation of CBC}
\label{subsec:implementation}
In general, implementation of the functional Steiner point of the work function may be computationally inefficient.  However, by taking advantage of the LTV structure, we are able to design an efficient implementation in our setting. The key observation here is that for each $t$, $\mathcal{P}_t$ (Algorithm~\ref{alg:main}, line~\ref{algoline:set}) can be described by the intersection of half-spaces because the ambient parameter space $\Theta$ is assumed to be a polytope and the observed online transition data from $x_t,\, u_t$ to $x_{t+1}$ specifies two half-space constraints at each time step due to linearity of \eqref{eq:ltv}. Our approach to approximate the functional Steiner point for chasing the consistent model sets is inspired by \cite{argue2021chasing} where second-order cone programs (SOCPs) are used to approximate the (nested set) Steiner point of the sublevel set of the work functions for chasing half-spaces.

Denote $\{(a_i,b_i)\}_{i=1}^{p_t}$ as the collection of $p_t$ half-space constraints describing $\mathcal{P}_t$, i.e., $a_i^{\top} \theta \leq b_i$. To approximate the integral for the functional Steiner point \eqref{eq:steiner} of $\omega_t$, we sample $N$ number of random directions $v \in \mathbb{S}^{n-1}$, evaluate the Fenchel conjugate of the work function $\omega^*_t$ at each $v$ with an SOCP, and take the empirical average. Finally we project the estimated functional Steiner point back to the set of consistent model $\mathcal{P}_t \cap \Theta$. Even though the analytical functional Steiner point \eqref{eq:steiner} is guaranteed to be a member of the consistent model set, the projection step is necessary because we are integrating  numerically, which may result in an approximation that ends up outside of the set. We summarize this procedure in Algorithm~\ref{alg:cbc}. 
Specifically, given a direction $v \in \mathbb{S}^{n-1}$, the Fenchel conjugate of the work function at time step $t$ is 
\begin{align*}
    {\omega^*_t}(v) &= \inf_{x \in \RR^n} \omega_t(x)-\inner{x}{v} \\
    &= \min_{\substack{x \in \RR^n\\ q_s\in\mathcal{K}_s}} \sum_{s=1}^t \norm{ q_s - q_{s-1} } + \norm{x - q_t} - \inner{x}{v}.
\end{align*}
This can be equivalently expressed as the following SOCP with decision variables $x,q_1,\ldots, q_t,\lambda, \lambda_1,\ldots, \lambda_t$:
\begin{equation}
\label{eq:socp}
\begin{aligned}
&\min_{\substack{x,q_1,\ldots, q_t \\ \lambda, \lambda_1,\ldots, \lambda_t}}   \quad && \, \lambda + \sum_{s=1}^t \lambda_s - \inner{v}{x} \\
&\,\,\,\text{s.t.} && \norm{q_s - q_{s-1}}  \leq \ \lambda_s,  \quad \text{for }s\in[t] \\
& && \norm{x - q_t}  \leq \ \lambda & \\
& && a_i^{\top}q_s \leq b_i,\quad \text{for }i\in [p_s],\, s\in[t]
\end{aligned}
\end{equation}

Another potential implementation challenge is that the number of constraints in the SOCP \eqref{eq:socp} grows linearly with time due to the construction of the work function \eqref{eq:work-function}. This is a common drawback of online control methods based on CBC and NCBC techniques and can be overcome through truncation or over-approximation in of the work functions in practice.  Additionally, if the LTV system is periodic with a known period, then we can leverage Algorithm~\ref{alg:main} during the initial data collection phase. Once representative (persistently exciting) data is available, one could employ methods like \cite{nortmann2020data} to generate a stabilizing controller for the unknown LTV system. In \Cref{sec:simulation}, we show that data collection via Algorithm~\ref{alg:main} results in a significantly smaller state norm than random noise injection when the system is unstable.

\section{Simulation}
\label{sec:simulation}
In this section, we demonstrate Algorithm~\ref{alg:main} in two LTV systems. Both of the systems we consider are open-loop unstable, thus the algorithms must work to stabilize them. We use the same algorithm parameters for both, with $\Theta = [-2,\,3]^2$, LQR cost matrices $Q = I$ and $R = 1$. 

\subsection{Example 1: Markov linear jump system}
We consider the following Markov linear jump system (MLJS) model from \cite{xiong2006stabilization}, with
$$
\begin{aligned}
& A_1=\left[\begin{array}{cc}
1.5 & 1 \\
0 & 0.5
\end{array}\right], \quad A_2=\left[\begin{array}{cc}
0.6 & 0 \\
0.1 & 1.2
\end{array}\right], \quad B_1=\left[\begin{array}{l}
0 \\
1
\end{array}\right], \\
& B_2=\left[\begin{array}{l}
1 \\
1
\end{array}\right], \quad \Pi=\left[\begin{array}{ll}
0.8 & 0.2 \\
0.1 & 0.9
\end{array}\right]
\end{aligned}
$$
where $\Pi$ is the transition probability matrix from $\theta_1$ to $\theta_2$ and vice versa. We inject uniformly random disturbances such that $w_t \in \{-10\mathds{1},\, -3\mathds{1},\, 3\mathds{1}\}$ where $\mathds{1}$ is the all-one vector. We set the disturbances to be zero for the last 10 time steps to make explicit the stability of the closed loop. We implement certainty-equivalent control based on online least squares (OLS) with different sliding window sizes $L =5,\,10,\,20$ and a exponential forgetting factor of $0.95$ \cite{jiang2004revisit} as the baselines.

We show two different MLJS models generated from 2 random seeds and show the results in Figure~\ref{fig:mljs}. For both systems, the open loop is unstable. In Figure~\ref{fig:mljs_1} the OLS-based algorithms fail to stabilize the system for window size of $L=20$, while stabilizing the system but incurring larger state norm than the proposed algorithm for $L=5,\,10$. On the other hand, in Figure~\ref{fig:mljs_2}, OLS with $L=5$ results in unstable closed loop. This example highlights the challenge of OLS-based methods, where the choice of window size is crucial for the performance. Since the underlying LTV system is unknown and our goal is to control the system \textit{online}, it is unclear how to select appropriate window size to guarantee stability for OLS-based methods a priori. In contrast, Algorithm~\ref{alg:main} does not require any parameter tuning.
 
We note that while advanced least-squares based identification techniques that incorporate sliding window with \textit{variable} length exist, e.g. \cite{luo2022dynamic, jiang2004revisit}, due to the unknown system parameters, it is unclear how to choose the various algorithm parameters such as thresholds for system change detection. Therefore, we only compare Algorithm 1 against fixed-length sliding window OLS methods as baselines.

\begin{figure}[h]
     \centering
     \begin{subfigure}[a]{0.4\textwidth}
         \centering
         \includegraphics[width=\textwidth, trim = 1cm 0.66cm 0cm 0.4cm, clip] 
         {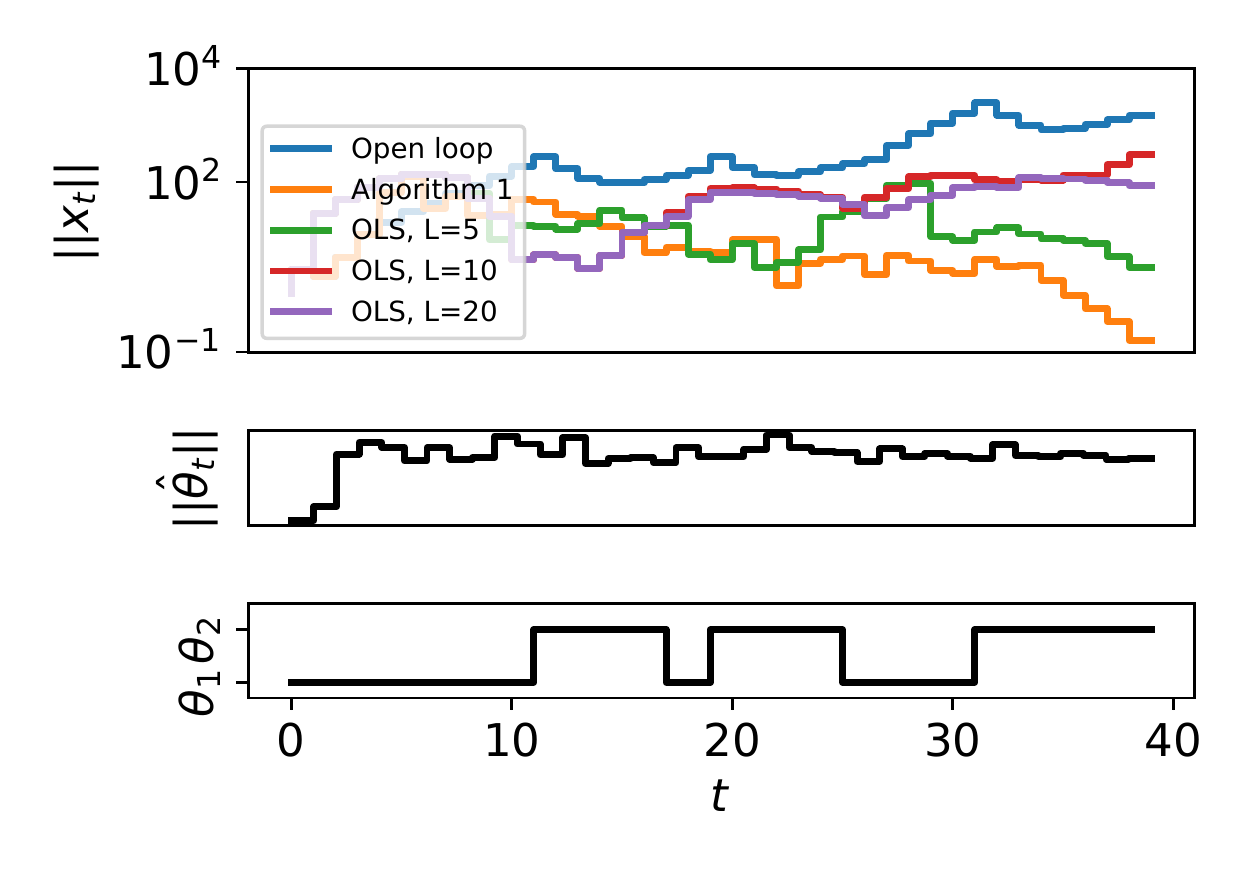}
         \caption{closed loop of the system generated with seed \# 1}
         \label{fig:mljs_1}
     \end{subfigure}
     \begin{subfigure}[b]{0.4\textwidth}
         \centering
         \includegraphics[width=\textwidth,trim = 1cm 0.65cm 0cm 0cm, clip]{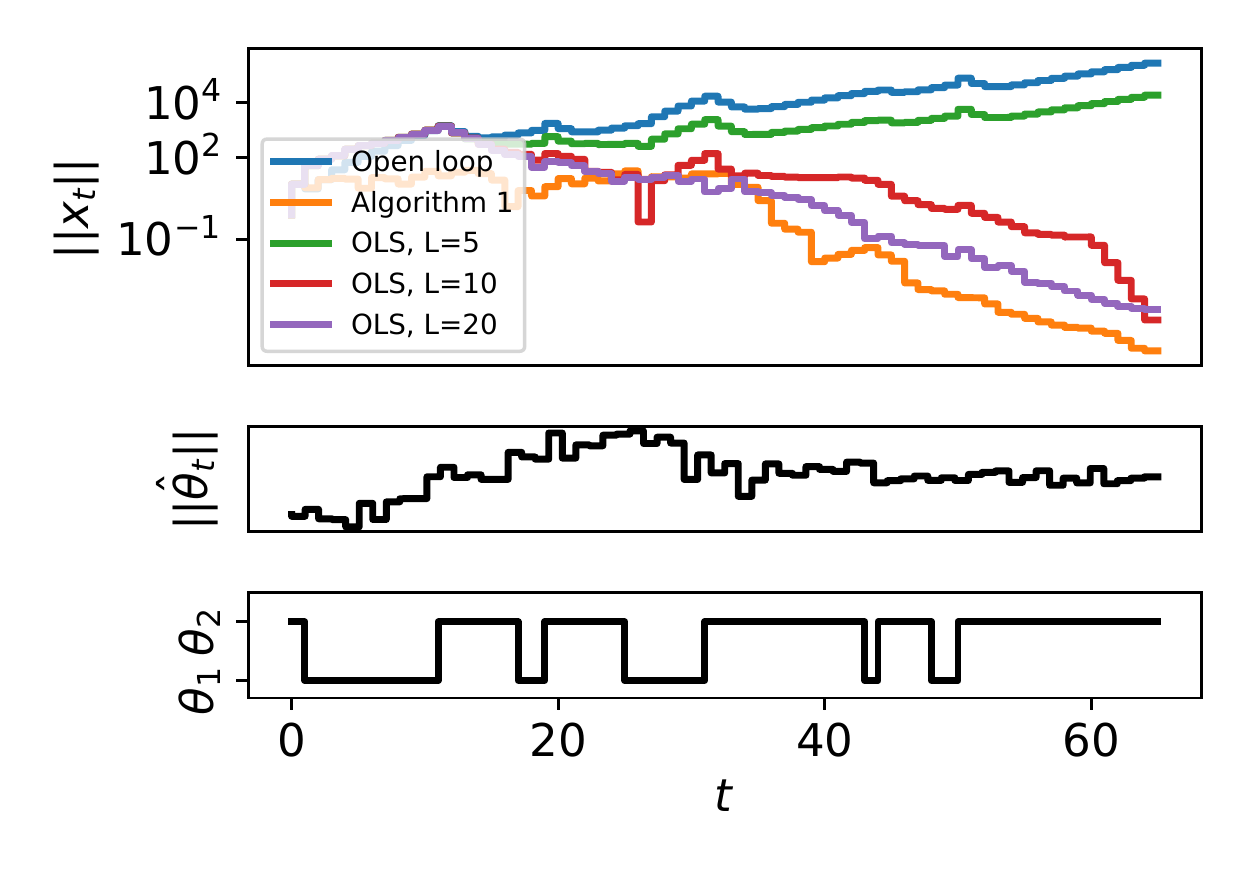}
         \caption{closed loop of the system generated with seed \# 2}
         \label{fig:mljs_2}
     \end{subfigure}
        \caption{Markov linear jump system for two different random seeds. For each seed: \textbf{Top} plot shows the state norm trajectories of the proposed algorithm, certainty-equivalent control based on online least squares (OLS) with different sliding window sizes, and the open loop. \textbf{Middle} plot shows the norm of the selected hypothesis model via Algorithm~\ref{alg:cbc}. \textbf{Bottom} plot shows the true model switches. }
         \label{fig:mljs}
\end{figure}


\vspace{-0.5cm}
\subsection{Example 2: LTV system }
Our second example highlights that Algorithm~\ref{alg:main} is a useful data-collection alternative to open-loop random noise injection. We consider the LTV system from \cite{nortmann2020data,pang2018data}, with
$$
\begin{aligned}
& A(k)=\left[\begin{array}{cc}
1.5 & 0.0025 k \\
-0.1 \cos (0.3 k) & 1+0.05^{3 / 2} \sin (0.5 k) \sqrt{k}
\end{array}\right], \\
& B(k)=0.05\left[\begin{array}{c}
1 \\
\frac{0.1 k+2}{0.1 k+3}
\end{array}\right] .
\end{aligned}
$$
where we modified $A(1,1)$ from 1 to 1.5 to increase the instability of the open loop in the beginning; thus making it more challenging to stabilize. We consider no disturbances here, which is a common setting in direct data-driven control, e.g., \cite{nortmann2020data, rotulo2022online, baros2022online}. In particular, we compare the proposed algorithm against randomly generated bounded inputs from $\textsf{UNIF}[-1,1]$. We also modify the control inputs from Algorithm~\ref{alg:main} to be $u_t = \hat{K}_{t-1} x_t + \eta_t\cdot\mathds{1}$ with $\eta_t \sim \textsf{UNIF}[-1,1]$ so that we can collect rich data in the closed loop. This is motivated by the growing body of data-driven control methods such as \cite{nortmann2020data, baros2022online, pang2018data} that leverage sufficiently rich offline data to perform control design for unknown LTV systems. However, most of these works directly inject random inputs for data collection. It is evident in Figure~\ref{fig:ltv} that when the open-loop system is unstable it may be undesirable to run the system without any feedback control. Therefore, Algorithm~\ref{alg:main} complements existing data-driven methods by allowing safe data collection with significantly better transient behavior.
\begin{figure}
    \centering
    \includegraphics[width=0.8\columnwidth, trim = 0.1cm 0.66cm 0.1cm 0.4cm, clip]{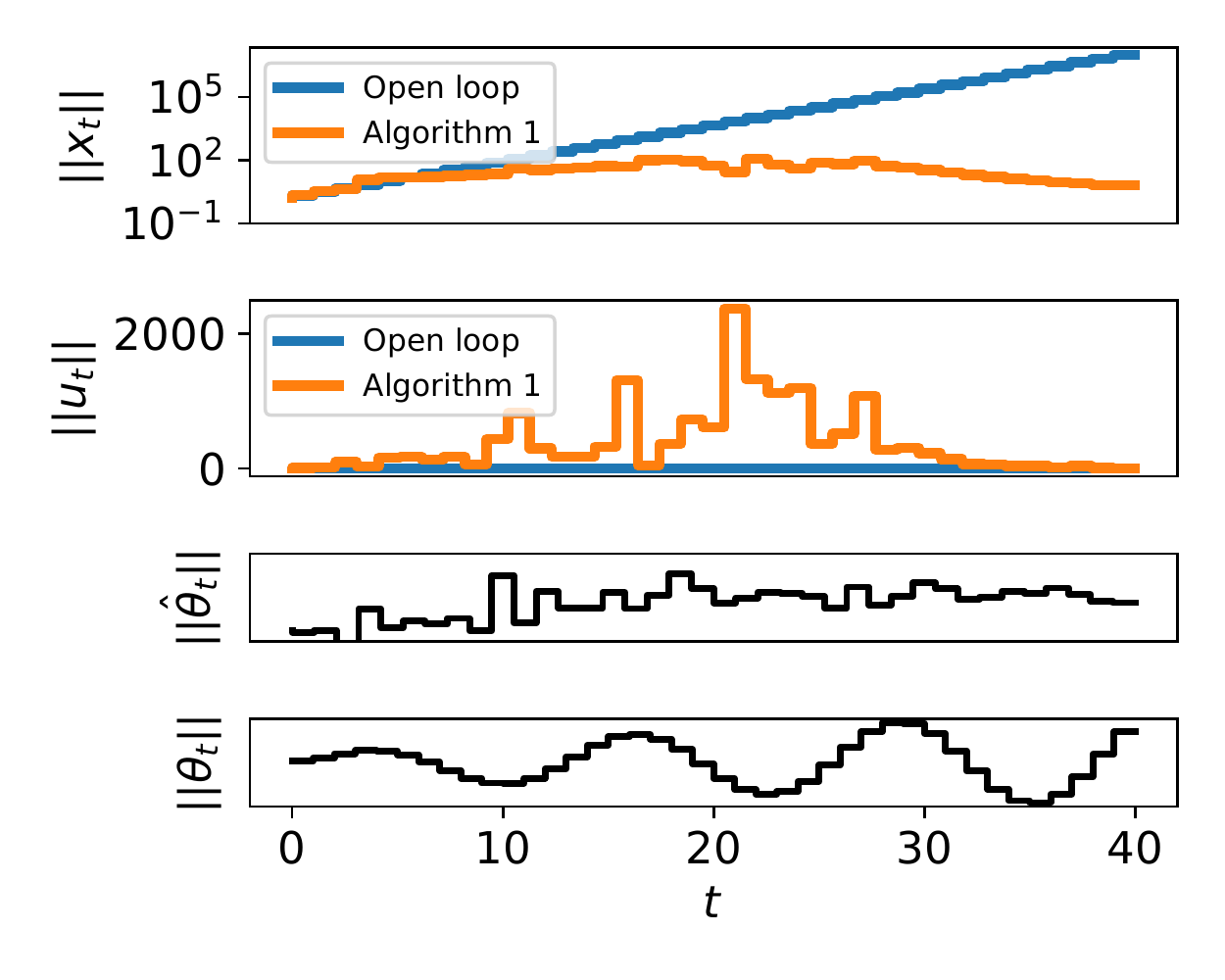}
    \caption{Simulation result for the LTV system in example 2. Here we plot the the state and control norm, as well as the selected hypothesis model via CBC $\hat \theta_t$ and true models $\theta_t$.   }
    \label{fig:ltv}
\end{figure}

\section{Concluding remarks}
\label{sec:conclusion}
In this paper, we propose a model-based approach for stabilizing an unknown LTV system under arbitrary non-stochastic disturbances in the sense of bounded input bounded output under the assumption of infrequently changing or slowly drifting dynamics. Our approach uses ideas from convex body chasing (CBC), which is an online problem where an agent must choose a sequence of points within sequentially presented convex sets with the aim of minimizing the sum of distances between the chosen points. The algorithm requires minimal tuning and achieves significantly  better performance than the naive online least squares based control. Future work includes sharpening the stability analysis to go beyond the BIBO guarantee in this work, which will require controlling the difference between the estimated disturbances and true disturbances. Another direction is to extend the current results to the networked case, similar to \cite{yu2023online}.

\bibliographystyle{IEEEtran}
\bibliography{reference.bib}

\appendix
\subsection{Proof of \Cref{lem:partial-path}}
\label{appedix:partial-path}
We have
\begin{align}
    \sum_{\tau=s+1}^{e} \norm{ \widehat{\theta}_{\tau} -  \widehat{\theta}_{\tau-1} }_F &= \sum_{\tau=s+1}^{e} \norm{\st(\omega_{\tau}) - \st(\omega_{\tau-1})}_F  \nonumber \\
    &\stackrel{(a)}{\leq}  n \dashint_{v} \left(\sum_{\tau=s+1}^{e} \abs{\omega^*_{\tau}(v) - \omega^*_{\tau-1}(v)}\right) \,v\,dv \nonumber \\
    &\stackrel{(b)}{=} n \dashint_{v} \left(\sum_{\tau=s+1}^{e} \omega^*_{\tau}(v) - \omega^*_{\tau-1}(v) \right) \,v\,dv \nonumber \\
    & =  n \dashint_{v} \left(\omega^*_{e}(v) - \omega^*_{s}(v) \right)\,v\,dv \nonumber \\
    &\stackrel{(c)}{\leq} n \cdot ( \min_{x} \omega_e(x) - \min_{y} \omega_s(y) + 2 \kappa) \label{eq:intermediate-step}
\end{align}
where (a) is due to the definition \eqref{eq:steiner}. For (b), we used the observation that $\omega^*_t(v)$ is non-decreasing in time. For (c), by definition of the Fenchel conjugate \eqref{eq:conjugate}, we have that $\omega^*_e(v) = \text{inf}_{x} \omega_e(x) -\inner{x}{v}$. Denote $(x^\star, q_1^\star,\ldots,q_e^\star)$ as the optimal solution to the problem $\min_x \omega_e(x)$. It is clear that $\omega^*_e(v) \leq  \omega_e(x^\star) -\inner{x^\star}{v} \leq \min_x \omega_e(x) + \kappa$ where in the last inequality we used Cauchy-Shwarz and $\kappa := \max_{\theta \in \Theta} \norm{\theta}_F$. Similarly, we also have $\omega_s^*(v) \geq \inf_y \omega_s(y) - \kappa$.

Denote $\OPT{e}$ as the minimizing trajectory $(\textsc{OPT}_0,\, \ldots, \textsc{OPT}_e)$ to $\min_x \omega_e(x)$ where $\text{argmin}_x \omega_e(x) = \textsc{OPT}_e$. This last equality is by the observation that if $x^\star:= \text{argmin}_x \omega_e(x) \not = \textsc{OPT}_e$, then $\omega_e(\textsc{OPT}_e) \leq \omega_e(x^\star)$ by definition \eqref{eq:work-function}, thus contradicting that $x^\star$ is defined to be the minimizer of $\omega_e$. 
We also denote $\INT{s}$ as the minimizing trajectory to $\min_y \omega_s(y)$. To reduce notation, we denote $\dOPT{s}{e}:=\sum_{\tau=s+1}^e \norm{\textsc{OPT}_{\tau}-\textsc{OPT}_{\tau-1} }_F$ and $\Delta^{\textsc{INT}}_{[s,e]} :=\sum_{\tau=s+1}^e \norm{\textsc{INT}_{\tau}-\textsc{INT}_{\tau-1} }_F$.
Then we have
\begin{align*}
    \eqref{eq:intermediate-step} &= n \cdot \left( \dOPT{0}{e} - \dINT{0}{s} + 2 \kappa\right)\\
    &\stackrel{(c)}{\leq} n \cdot \left( \dOPT{0}{e} - \dOPT{0}{s} + \textsf{dia}(\Theta) + 2 \kappa \right)\\
    &= n \cdot \left(\dOPT{s}{e}  + \textsf{dia}(\Theta) + 2 \kappa\right).
\end{align*}
where (c) holds because if $\sum_{\tau=1}^{s} \norm{\textsc{OPT}_{\tau}-\textsc{OPT}_{\tau-1} }_F  > \sum_{\tau=1}^{s} \norm{\textsc{INT}_{\tau}-\textsc{INT}_{\tau-1} }_F + \textsf{dia}(\Theta)$ and $\textsc{OPT}_{[0,s]} \not = \textsc{INT}_{[0,s]} $, then we can replace the $[0,s]$ portion of the optimal trajectory $\OPT{e}$ with $\INT{s}$ and achieve a lower cost for $\omega_e(\textsc{OPT}_e)$, thus contradicting the optimality of $\OPT{e}$. To see why the fictitious trajectory $\left(\textsc{INT}_{[0,s]}, \textsc{OPT}_{[s+1,e]}\right)$ achieves lower cost than $\textsc{OPT}_{[0,e]}$, we compare the total movement cost during the interval $[0,s+1]$,
\begin{align*}
    &\sum_{\tau=1}^s \norm{\textsc{INT}_\tau - \textsc{INT}_{\tau-1}}_F + \norm{\textsc{OPT}_{s+1} - \textsc{INT}_{s}}_F\\ &\leq \sum_{\tau=1}^s \norm{\textsc{INT}_\tau - \textsc{INT}_{\tau-1}}_F + \norm{\textsc{OPT}_{s+1} - \textsc{OPT}_{s}}_F\\ & \qquad + \norm{\textsc{OPT}_{s} - \textsc{INT}_{s}}_F\\
    &\leq \sum_{\tau=1}^s \norm{\textsc{INT}_\tau - \textsc{INT}_{\tau-1}}_F + \norm{\textsc{OPT}_{s+1} - \textsc{OPT}_{s}}_F + \textsf{dia}(\Theta)\\
    &< \sum_{\tau=1}^s \norm{\textsc{OPT}_\tau - \textsc{OPT}_{\tau-1}}_F + \norm{\textsc{OPT}_{s+1} - \textsc{OPT}_{s}}_F, 
\end{align*}
which means the fictitious trajectory achieves lower overall cost. Therefore (c) must hold.
\hfill $\blacksquare$

\subsection{Auxiliary results}
\label{appendix:auxillary}

Here we summarize the helper lemmas used in the proof sketch of \Cref{thrm:iss}. First, we define some useful notation.

\noindent\textbf{Lyapunov equation.} Let $X,Y \in \RR^{n \times n}$ with $Y = Y^\top \succ 0$ and $\rho(X) < 1$. Define $\dlyap(X,Y)$ to be the unique positive definite solution $Z$ to the Lyapunov equation $X^\top Z X - Z = Y$. 
For a stabilizable system $(A,B)$ with optimal infinite-horizon LQR feedback $K:=K^*( [ A \ B])$ with cost matrices $Q, R = I$, we define 
$$P(A,B) = \dlyap(A+BK^*( [ A \ B]), \,I_n + K^*( [ A \ B])^\top K^*( [ A \ B]) )$$ and 
$$H(A,B) = \dlyap(A+BK^*( [ A \ B]),\, I_n).$$ 
We also define the shorthand for the following:
\begin{equation}
\label{eq:PandH}
   {P}_t := P(\widehat{A}_t, \widehat{B}_t), \quad {H}_t := H(\widehat{A}_t, \widehat{B}_t).
\end{equation}

\noindent \textbf{Constants. } Throughout the proof, we will reference the following system-theoretical constants for the parameter set $\Theta$ defined in \Cref{assum:compact_set}:
    \begin{align*}
    &\norm{K_*} := \sup_{ [A \ B]\in \Theta} \norm{K^*([A \ B])}, \gamma_*:= \max_{ [A \ B]\in\Theta} \norm{A+BK^*([A \ B])}.
    \end{align*}
We also quantify the stability of every model in $\Theta$ under its corresponding optimal LQR gain. Let $$C_*>0, \quad  r_*\in(0,1)$$ be
such that for all $\theta :=[A \ B]\in \Theta$, $K:=K^*(\theta)$, and $i\in \mathbb{N}_+$,
$ \norm{\left( \left( A + BK\right)^T\right)^i} \cdot \norm{\left( A + BK\right)^i}  \leq C_* r_*^{2i} $.
By \Cref{lemma:householder} which is stated below and \Cref{assum:compact_set}, such $C_*$ and $r_*$ always exist. 
Further, we define 
\begin{align*}
\norm{P_*} &:= \sup_{ [A \ B]\in \Theta} \norm{P(A,B)},\quad \norm{H_*} := \sup_{ [A \ B]\in \Theta} \norm{H(A,B)}, \\
\epsilon_* &:= 1/\left(54 \norm{P_*}^5\right), \quad c_*:=  \max_{ [A \ B]\in\Theta} \frac{\lambda_{\max}H(A,B)}{\lambda_{\min}H(A,B)},\\
  h_* &:=\sup_{ [A_1 \ B_1],\, [A_2 \ B_2] \in \Theta} \norm{H(A_1,B_1)^{1/2}}\norm{H(A_2,B_2)^{-1/2}}.
\end{align*}
To justify the existence of these constants, note that discrete-time optimal LQR controller has guaranteed stability margin \cite{shaked1986guaranteed} and that by \Cref{lemma:householder} and the fact that the solution to Lyapunov equation has the following closed form, 
\begin{equation}
\label{eq:lyapuv}
P(A,B) = \sum_{i=0}^\infty \left((A+BK)^\top\right)^i (I+K^\top K)(A+BK)^i,
\end{equation}
we have that for all $[A,B]\in \Theta$,
\begin{align*}
    \norm{P(A,B)} & \leq \left( 1 + \norm{K}^2 \right) \left( 1 + \sum_{i=1}^{\infty} \norm{ \left( \left(A + BK\right)^\top \right)^i } \norm{ \left(A+BK \right)^i } \right)\\
    & \leq \frac{\left(1 + \norm{K_*}^2\right) \left(1-r_*^2 + C_*\right)}{1-r_*^2} =: \norm{P_*}.
\end{align*}
We can similarly derive $ \norm{H_*}$. 
By definition of the Lyapunov solution \eqref{eq:lyapuv}, $\norm{P_*}\geq \norm{H_*} \geq 1$.

\begin{lemma}[{\protect{\cite[page 183]{householder2013theory}}}]
\label{lemma:householder}
For a matrix $A \in \RR^{n \times n}$, with $\rho := \rho(A)$, there exist constants $\kappa_1,\kappa_2$ such that for any positive integer $i$
\[ \kappa_1 \rho^i i^{n_1-1}  \leq  \norm{A^i} \leq \kappa_2 \rho^i i^{n_1-1} \]
where $n_1$ is the size of the largest Jordan block corresponding to eigenvalue of $\rho$ in Jordan block form representation of $A$.
\end{lemma}

\begin{lemma}[{\protect{\cite[Proposition 6]{simchowitz2020naive}}}]
\label{lem:BK_bound}
Let $\Theta = [A \  B]$ be a stabilizable system, with optimal controller $K := K^*( \theta)$ and $P := P(A,B)$. Let $\hat{\theta} = [ \hat{A} \ \hat{B} ]$ be an estimate of $\theta$, $\hat{K} := K^*(\hat{\theta})$ the optimal controller for the estimate,
and 
$\epsilon:= \max\left\{\norm{A - \hat{A}}, \norm{B - \hat{B}} \right\}$. Then if $\alpha := 8 \norm{P}^2 \epsilon < 1$:
\[ \norm{ B \left(  \hat{K} - K \right) } \leq 8 (1-\alpha)^{-7/4} \norm{P}^{7/2} \epsilon.
\]
\end{lemma}

\begin{lemma}
[{\protect{\cite[Theorem 8]{simchowitz2020naive}}}]
\label{lem:AHA_bound}
Let $\theta = [A \  B]$ be a stabilizable system, with $P := P(A,B)$, and $H = H(A, B)$. Let
$\hat{\theta} = [ \hat{A} \ \hat{B} ]$ be an estimate of $\theta$ satisfying
$\max\left\{\norm{A - \hat{A}}, \norm{B - \hat{B}} \right\}\leq \epsilon$. Consider certainty equivalent controller $\hat{K} =K^*(\hat{\theta})$.
Then if $\epsilon$ is such that $ 54 \norm{P}^5 \epsilon \leq 1$, we have
\begin{align*}
 (A+B\hat{K})^\top H (A+B\hat{K}) &\preceq \left( 1 - \frac{1}{2} \norm{H}^{-1}\right) H \preceq \left( 1 - \frac{1}{2}\norm{P}^{-1}\right) H. 
\end{align*}
\end{lemma}

\begin{lemma}[\cite{gahinet1990computable}]
\label{lemma:sensitivity}
Let $X$ be the solution to the Lyapunov equation $X-F^{\top} X F=M$, and let $X+\Delta X$ be the solution to the perturbed problem
$$
Z-(F+\Delta F)^{\top} Z (F+\Delta F)=M.
$$
The following inequality holds for the spectral norm:
$$
\frac{\|\Delta X\|}{\|X+\Delta X\|} \leq 2 \left\|\sum_{k=0}^{+\infty}\left(F^{\top}\right)^{k} F^{k}\right\|\cdot (2\|F\|+\|\Delta F\|) \cdot {\|\Delta F\|}.
$$
\end{lemma}

\begin{lemma}
\label{lem:HH_bound}
Suppose $\epsilon_{t+1}:=\max\left\{ \norm{ \widehat{A}_{t+1} - \widehat{A}_t } , \norm{ \widehat{B}_{t+1} - \widehat{B}_t } \right\}$ and $\alpha :=  8 \norm{P_*}^2 \epsilon_{t+1} \leq 1/2$,  . Then $H_t$ defined in \eqref{eq:PandH} satisfies 
$$H_t \preceq H_{t+1}(1 + \eta_{t+1})$$ 
for $\eta_{t+1} :=  c_* \beta_* \epsilon_{t+1}$, and
\[ \beta_*:= \frac{2C_*}{1-r_*^2} \left( 2\gamma_* + 3 + \norm{K_*} \right)  \left( 1 + 32 \norm{P_*}^2 + \norm{K_*} \right).\]
\end{lemma}

\textit{Proof of \Cref{lem:HH_bound}. } For notational brevity, we drop the time index for $\epsilon$ and $\eta$ in the proof.
Applying Lemma~\ref{lemma:sensitivity} with $X = H_{t}$, $X+\Delta X  = H_{t+1}$ and $F = \widehat{A}_t+\widehat{B}_t\widehat{K}_{t}$ and $\Delta F = (\widehat{A}_{t+1} - \widehat{A}_{t}) + (\widehat{B}_{t+1}\widehat{K}_{t+1} - \widehat{B}_{t}\widehat{K}_{t}) $, and $M=I_n$ we have 
\begin{align*}
&\frac{\norm{ H_{t+1} - H_t} }{ \norm{H_{t+1}}}  \leq 2 \left\|\sum_{k=0}^{+\infty}\left((\widehat{A}_t+\widehat{B}_t \widehat{K}_{t})^{\top}\right)^{k} (\widehat{A}_t+\widehat{B}_t \widehat{K}_{t})^{k}\right\|\\
&\quad\cdot \Big(2\norm{\widehat{A}_t+\widehat{B}_t \widehat{K}_{t}}+ \norm{ \widehat{A}_{t+1} - \widehat{A}_{t})} +\\
& \quad \quad \quad \quad \quad \norm{\widehat{B}_{t+1} (\widehat{K}_{t+1}-\widehat{K}_{t} ) } + \norm{(\widehat{B}_{t+1}- \widehat{B}_{t})\widehat{K}_{t} } \Big) \\
& \quad \cdot \Big( \norm{ \widehat{A}_{t+1} - \widehat{A}_{t})} + \norm{\widehat{B}_{t+1} (\widehat{K}_{t+1}-\widehat{K}_{t} ) } + \norm{(\widehat{B}_{t+1}- \widehat{B}_{t})\widehat{K}_{t} }  \Big)\\
& \quad   \quad  \leq \epsilon \frac{2C_*}{1-r_*^2} \left( 2\gamma_* +  
 \epsilon \left(  1 + 32 \norm{P_*}^2 + \norm{K_*} \right) \right)\cdot \\
 &\quad \quad \quad \quad \quad \quad \quad \quad\left(  1 + 32 \norm{P_*}^2 + \norm{K_*} \right)  \\
 & \quad \quad \quad \quad \quad  \leq \epsilon \frac{2C_*}{1-r_*^2} \left( 2\gamma_* + 3 + \norm{K_*} \right)  \left( 1 + 32 \norm{P_*}^2 + \norm{K_*} \right)\\
 & \quad \quad \quad \quad \quad =: \epsilon \beta,
\end{align*}
where in the second inequality we used Lemma~\ref{lem:BK_bound} to bound $\norm{\widehat{B}_{t+1}( \widehat{K}_{t+1} - \widehat{K}_t) } \leq 32 \norm{P_{t+1}}^{7/2} \epsilon$ and in the last inequality we use the assumption $8 \epsilon \norm{P_*}^2 \leq 1/2$.

To show $H_t \preceq H_{t+1}(1 + \eta)$ for some $\eta$, it suffices to show that for all vectors $v \in RR^n$, 
$v^\top (H_t - H_{t+1}) v \leq \eta v^\top H_{t+1} v$. 
 With the preceding calculation, we have 
\begin{align*}
     v^\top (H_t - H_{t+1})v &\leq \| v\|^2 
 \norm{H_t - H_{t+1}}\\
 &\leq \epsilon \beta_* \| v\|^2  \norm{H_{t+1}}\\
 & \leq \epsilon \beta_* c_* \lambda_{\min}(H_{t+1})\| v\|^2\\
 &\leq \epsilon \beta_* c_*  v^\top H_{t+1}v
\end{align*}
This proves the desired bound, with $\eta = c_* \beta_* \epsilon$ and$$\beta_* =  \frac{2C_*}{1-r_*^2} \left( 2\gamma_* + 3 + \norm{K_*} \right)  \left( 1 + 32 \norm{P_*}^2 + \norm{K_*} \right).$$
$\hfill \blacksquare$

\subsection{Proof of \Cref{thrm:iss}}
\label{appendix:main_proof}

Recall that the closed loop dynamics can be characterized as \eqref{eq:closed-loop}. Therefore, 
\begin{align}
\label{eq:closed-loop-norm}
    \norm{x_t} \leq  W + W\sum_{s = 0}^{t-2} \norm{ \prod_{\tau \in [t-1 : s+1] } \left( \widehat{A}_{\tau} + \widehat{B}_{\tau} \widehat{K}_{\tau-1} \right)  }.
\end{align}
Define 
\begin{align*}
    L_t &:= H_t^{-1/2} (\widehat{A}_t + \widehat{B}_t \widehat{K}_{t-1}) H_t^{1/2},
\end{align*}
where $H_t$ is defined in \eqref{eq:PandH}. This gives,
\begin{align*}
    \widehat{A}_t + \widehat{B}_t \widehat{K}_{t-1} &:= H_t^{1/2}  L_t  H_t^{-1/2}.
\end{align*}
Therefore, each summand in \eqref{eq:closed-loop-norm} can be bounded as
\begin{align}
    & \norm{ \prod_{\tau\in I_s} \left(\widehat A_\tau + \widehat B_\tau \widehat K_{\tau-1}\right) } \nonumber \\
    &\leq \underbrace{\norm{H_{t-1}^{1/2}}\norm{ H_{s+1}^{-1/2}}}_{(a)}
        \underbrace{\prod_{k\in I_{s+1}} \norm{H_{k}^{-1/2}  H_{k-1}^{1/2}}}_{(b)}   
         \underbrace{\prod_{\tau\in I_s} \norm{L_{\tau}}}_{(c)} \label{eq:summand-norm}
\end{align}
where we used $I_s$ as shorthand for the interval $[t-1:s+1]$.

\noindent \textbf{Bounding (a). }
We directly use the system-theoretical constant introduced in \Cref{appendix:auxillary} so that (a) $\leq h_*$.

\noindent \textbf{Bounding (b). }
Lemma~\ref{lem:HH_bound} directly implies that for all $t\in \mathbb{N}_+$, $H_{t-1} H_t^{-1} \succ (1+\eta_t)I$. Therefore, we have
\begin{align*}
 \norm{H_{t-1}^{1/2}H_{t}^{-1/2}} &\leq (1+\eta_t)^{1/2} \leq 1 + \eta_t/2 \leq e^{\eta_t/2}.
\end{align*}
Hence with the fact that $H_t$'s are symmetric,
\begin{align}
\label{eqn:bound_HH}
    \norm{H_{t}^{-1/2}H_{t-1}^{1/2}} &\leq
    \begin{cases}
        e^{\frac{c_*\beta_* \norm{\hat \theta_t - \hat \theta_{t-1}}_F}{2}}, &  \norm{\hat \theta_t - \hat \theta_{t-1}}_F\leq \epsilon_* \\
        h_* & \mbox{otherwise}.
    \end{cases}
\end{align}

\noindent \textbf{Bounding (c). }
Lemma~\ref{lem:AHA_bound} implies that if $\norm{\hat \theta_t - \hat \theta_{t-1}}_F \leq \epsilon_*$  then
\[  \left( \widehat{A}_t + \widehat{B}_t \widehat{K}_{t-1} \right)^\top H_t \left( \widehat{A}_t + \widehat{B}_t \widehat{K}_{t-1} \right) \preceq \left(  1 - \frac{1}{2} \norm{P_t}^{-1}  \right) H_t. \]
This in turn implies that
\begin{align*}
    L_t^\top L_t &= H_t^{-1/2} (\widehat{A}_t + \widehat{B}_t \widehat{K}_{t-1})^\top H_t (\widehat{A}_t + \widehat{B}_t \widehat{K}_{t-1}) H_t^{-1/2} \\
    & \preceq H_t^{-1/2} \left(  1 - \frac{1}{2} \norm{P_t}^{-1}  \right) H_t H_t^{-1/2} \\
    & \preceq \left(  1 - \frac{1}{2} \norm{P_t}^{-1}  \right) I_n. 
\end{align*}
This in turn implies that 
$\norm{L_t} \leq \left( 1 - \frac{1}{2\norm{P_*}}  \right)^{1/2}$.
To summarize, 
\begin{align}
\label{eqn:bound_L}
    \norm{L_t} & \leq 
    \begin{cases}
    \rho_L := \left( 1 - \frac{1}{2\norm{P_*}}  \right)^{1/2} <1, &  \norm{\hat \theta_t - \hat \theta_{t-1}}_F\leq \epsilon_*\\
    \ell_* & \mbox{otherwise},
    \end{cases}
\end{align}
for some constant $\ell_*$ such that for all $t\in \mathbb{N}_+$, 
$$\norm{H_t^{1/2} (\widehat{A}_t + \widehat{B}_t \widehat{K}_{t-1}) H_t^{-1/2}} \leq \ell_*$$

\noindent \textbf{Combining (a,b,c). }
We now plug in the bounds \eqref{eqn:bound_HH} and \eqref{eqn:bound_L} into  \eqref{eq:summand-norm}. Let $\hat{\Delta}_{[s,e]} := \sum_{\tau = s+1}^e \norm{\hat \theta_{\tau} - \hat \theta_{\tau -1}}_F$ be the partial-path movement of the selected hypothesis models and ${\Delta}_{[s,e]} := \sum_{\tau = s+1}^e \norm{ \theta_{\tau} -  \theta_{\tau -1}}_F$ be the true model  partial-path variation. 
We also denote by $n_{s,t}$ the number of pairs $(\tau,\tau-1)$ with $s+1 \leq \tau \leq t-1$ where $\norm{\hat \theta_\tau - \hat \theta_{\tau-1}}_F> \epsilon_*$. Note that $n_{s,t} \leq \widehat \Delta_{[s,t-1]} / \epsilon_*$. Therefore,
\begin{align*}
    &  \norm{\prod_{\tau\in [t-1:s+1]} \left(\widehat A_\tau + \widehat B_\tau \widehat K_{\tau-1}\right) }\\
    & \leq h_* \cdot h_*^{n_{s,t}} \cdot e^{\frac{c_*\beta_* \hat{\Delta}_{[s+1,t-1]} }{2}} \cdot \ell_*^{n_{s,t}} \cdot \rho_L^{t-s-1-n_{s,t}}
    \\
    &\leq h_* \left( \frac{\ell_* h_*}{\rho_L} \right)^{\frac{\hat{\Delta}_{[s,t-1]}}{\epsilon_*}}  e^{\frac{c_*\beta_* \hat{\Delta}_{[s+1,t-1]} }{2}}\cdot  \rho_L^{t-s-1}\\
    &\leq h_* \left( \frac{\ell_* h_*}{\rho_L} \right)^{\frac{ 
 \bar{n} \left( \mathsf{dia}(\Theta)+ 2\kappa + {\Delta}_{[s,t-1]} \right) }{\epsilon_*}} \cdot e^{ \frac{c_* \beta_* \bar{n} \left( \mathsf{dia}(\Theta)+ 2\kappa + {\Delta}_{[s+1,t-1]} \right)}{2} } \cdot  \rho_L^{t-s-1}\\
    & =: c_0 \cdot c_1^{\Delta_{[s,t-1]}} \rho_L^{t-s-1}.
\end{align*}
where $\bar{n} := n(n + m)$ is the dimension of the parameter space for $[A_t \ B_t]$. 
Finally plugging the above in \eqref{eq:closed-loop-norm} gives
\begin{align*}
    \norm{x_{t}} &\leq W \left(1 + c_0 \sum_{s=0}^{t-2} c_1^{\Delta_{[s,t-1]}} \rho_L^{t-s-1}\right).
\end{align*}
 \hfill $\blacksquare$

\end{document}

%% file: cdc-preamble.tex
\usepackage{mathptmx}
\usepackage[libertine]{newtxmath}
\usepackage{epsfig}
\usepackage{amsmath}
\usepackage{amsfonts}
\usepackage{thmtools}
\usepackage{theorem}
\usepackage{comment}
\usepackage{xcolor}
\usepackage{subcaption}
\usepackage{comment}
\usepackage{mathtools}

\usepackage{dsfont}
\usepackage[vlined,ruled,resetcount]{algorithm2e}

\usepackage[usestackEOL]{stackengine}
\def\dashint{\,\ThisStyle{\ensurestackMath{%
  \stackinset{c}{.2\LMpt}{c}{.5\LMpt}{\SavedStyle-}{\SavedStyle\phantom{\int}}}%
  \setbox0=\hbox{$\SavedStyle\int\,$}\kern-\wd0}\int}

\allowdisplaybreaks

\usepackage[hidelinks]{hyperref}
\usepackage{cleveref}

\newcommand{\st}{\textsf{st}}

\newcommand{\RR}{\mathbb{R}}
\newcommand\abs[1]{\left|#1\right|}
\newcommand\norm[1]{\left\Vert#1\right\Vert}
\newcommand\infnorm[1]{\left\Vert#1\right\Vert_\infty}

\newtheorem{theorem}{Theorem}
\newtheorem{corollary}{Corollary}
\newtheorem{definition}{Definition}
\newtheorem{assumption}{Assumption}
\newtheorem{lemma}{Lemma}
\newtheorem{remark}{Remark}

\newcommand{\inner}[2]{\left\langle #1,#2 \right\rangle}

\newcommand{\dlyap}{\textsf{dlyap}}

\let\hat\widehat

\newcommand{\OPT}[1]{\textsc{OPT}_{[0,#1]}}
\newcommand{\INT}[1]{\textsc{INT}_{[0,#1]}}
\newcommand{\sysm}{\{\theta_t\}_{t=1}^{\infty}} 
\newcommand{\bigO}{\mathcal{O}}

\usepackage{scalerel}
\usepackage[usestackEOL]{stackengine}

\def\dashint{\,\ThisStyle{\ensurestackMath{%
            \stackinset{c}{.2\LMpt}{c}{.5\LMpt}{\SavedStyle-}{\SavedStyle\phantom{\int}}}%
        \setbox0=\hbox{$\SavedStyle\int\,$}\kern-\wd0}\int}

\newcommand{\dINT}[2]{\Delta^{\textsc{INT}}_{[#1,#2]} }
\newcommand{\dOPT}[2]{\Delta^{\textsc{OPT}}_{[#1,#2]} }

\makeatletter
\newcommand{\changeoperator}[1]{%
  \csletcs{#1@saved}{#1@}%
  \csdef{#1@}{\changed@operator{#1}}%
}
\newcommand{\changed@operator}[1]{%
  \mathop{%
    \mathchoice{\textstyle\csuse{#1@saved}}
               {\csuse{#1@saved}}
               {\csuse{#1@saved}}
               {\csuse{#1@saved}}%
  }%
}
\makeatother

\changeoperator{prod}
